\documentclass[12pt,a4paper,oneside,titlepage]{article}

\usepackage{amsmath} 
\usepackage{amsfonts} 
\usepackage{amsthm} 
\providecommand{\alaplus}{\genfrac{}{}{0pt}{}{}{+}} 
\providecommand{\aladots}{\genfrac{}{}{0pt}{}{}{\cdots}} 
\providecommand{\K}{\mathop{\mathchoice{\doK\LARGE}{\doK\Large}{\doK\small}{\doK\small}}}
\providecommand{\doK}[1]{\vcenter{#1\kern.2ex\hbox{\normalfont\text{K}}\kern.2ex}}
\providecommand{\Z}{\mathbb{Z}} 
\providecommand{\Q}{\mathbb{Q}} 

\newtheorem{theorem}{Theorem}[section] 
\newtheorem{lemma}[theorem]{Lemma}     
\newtheorem{corollary}[theorem]{Corollary}

\begin{document}

\title{Rational approximations of the exponential function at rational points}
\author{Kalle Lepp\"al\"a, Tapani Matala-aho and Topi T\"orm\"a \\ 11J82 (primary), 11J70, 11A55 (secondary)}

\maketitle

\begin{abstract}
We give explicit and asymptotic lower bounds for the quantity $|e^{s/t}-M/N|$ by studying
a generalized continued fraction expansion of $e^{s/t}$. 
In cases $|s|\geq 3$ we improve existing results by extracting a large common factor from the numerators and the denominators 
of the convergents of that generalized continued fraction.
\end{abstract}


\section{Introduction} 

We will present both explicit and asymptotic irrationality measure results i.e. lower bounds 
for the quantity $|e^{s/t}-M/N|$ as a function of positive integer $N$, where $s/t$ is an arbitrary non-zero rational number.

As usual, the lower bounds are achieved by constructing sequences of high quality rational approximations to $e^{s/t}$.
Our approximations are based on the convergents of the generalized continued fraction expansion
\begin{align}\label{GENCONTFRACEXP}
e^{z}=1+\frac{2z}{2-z}\alaplus\frac{z^2}{6}\alaplus\frac{z^2}{10}\alaplus\frac{z^2}{14}\alaplus\aladots 
=1+\cfrac{2z}{2-z+\K_{n=1}^{\infty}\frac{z^2}{4n+2}}\,,
\end{align}
which are in fact the diagonal Pad\'e approximants of the exponential function \cite{perron},\cite{joth}, \cite{khov}.
Special attention will be paid for divisibility properties of the numerators and the denominators of the convergents of
\[
\K_{n=1}^\infty \frac{z^2}{4n+2}
\]
evaluated at $z = s/t$.
There appears an unexpected and fairly big common factor which is related to the prime number decomposition of $s$.
Thereby when $|s| \ge 3$, we improve the existing results considerably;
the comparison will be done in the next section.

Let $s\in\Z\setminus\{0\}$, $t\in\mathbb{Z}_{\ge 1}$ and $\gcd(s,t)=1$.
Denote the inverse of the function 
$y(z) = z \log z$, $z \geq 1/e$, as $z(y)$.
We define $z_0(y) = y$ and $z_n(y) = y/\log z_{n-1}(y)$, $n \in \mathbb{Z}_{\ge 1}$.
In the following we denote
\[
\alpha = \prod_{\substack{p \in \mathbb{P}\\ p \mid s}} p^{1/(p-1)}\,, \qquad \beta =  \sum_{\substack{p \in \mathbb{P}\setminus\{2\}\\ p \mid s}} 1 \, , \qquad
\gamma = \prod_{\substack{p \in \mathbb{P}\setminus\{2\}\\ p \mid s}} p\,, 
\]
\[
\sigma  =\frac{4t\alpha}{es^2}\,, \qquad \rho =
\begin{cases}
\tfrac{7}{3} \,, \quad &\sigma \ge 1\,, \\
5-2\log \sigma\,, \quad & \sigma < 1 \,,
\end{cases}
\qquad
\zeta(N) = \frac{z(\sigma \log N)}{\sigma} +\beta\,,
\]
\begin{align*}
Z(N)&  = \left(\frac{s^2}{\alpha^2}(\zeta(N)+1)\right)^\beta \\ & \qquad \qquad \times \left(
\frac{8\gamma^2 |s|(4t\zeta(N)+6t+s^2)\zeta(N)^\beta}{\alpha} 
+\frac{\gcd(2,s)\alpha^2\gamma}{N^2 s^2|e^{s/t}-1|}\right)
\,,
\end{align*}
\[
\eta  =\max\left\{\frac{\sqrt{e}|e^{s/t}-1|\gamma }{\sqrt2\sigma^{\beta}} - \frac{1}{2},\frac{e}{\sigma}+\beta\right\}\,, \qquad \varepsilon(N)=\frac{\log\log\log N}{\log\log N}\,.
\]

\begin{theorem}\label{GENERALTHEOREM}
Let $s\in\Z\setminus\{0\}$, $t\in\mathbb{Z}_{\ge 1}$ and $\gcd(s,t)=1$. 
Then
\[
1<\left|e^{s/t}-\frac{M}{N}\right|Z(N)\,
N^{2+2\log(|s|/\alpha)\,z(\sigma\log N)/(\sigma\log N)}
\]
for all 
$M\in\mathbb{Z}\setminus\{0\}, N \in \mathbb{Z}_{\ge N_1}$
with
\begin{equation}\label{RAJAN}
\log N_1 = \max\left\{(\eta-\beta) \log\left(\sigma(\eta-\beta)\right),\log\left(\frac{\gcd(2,s)}{|4s+2(s-2t)(e^{s/t}-1)|}\right)\right\}\,.
\end{equation}
\end{theorem}

If we wish to use elementary functions only, we can approximate the function $z(y)$ by the function $z_1(y)$ to get a more familiar looking bound.

\begin{corollary}\label{SPECIALTHEOREM}
Let $s \in \mathbb{Z}\setminus\{0\}, t\in\mathbb{Z}_{\ge 1}$ and $gcd(s,t)=1$. 
Then
\[
1<\left|e^{s/t}-\frac{M}{N}\right| c_2 \,
N^{2+2\log(|s|/\alpha)\,(1+\rho\varepsilon(N))/\log\log N}\left(\frac{\log N}{\log\log N}\right)^{2\beta+1} \,, 
\]
\[
c_2 = \frac{8\gamma^2 |s|^{2\beta +1}(1+d)^{\beta}\left(4t+d(6t+s^2)\right)}{\alpha^{2\beta + 1}}  
 +\frac{\gcd(2,s)\gamma s^{2(\beta-1)}d(d+d^2)^{\beta}}{N_2^2\alpha^{2(\beta - 1)}|e^{s/t}-1|}\,,
\]
\[
d = \left( \frac{(\log N_2)^{1/4}}{\sigma}+\beta\right) ^{-1}
\]
for all $M \in\mathbb{Z}\setminus\{0\}, N \in \mathbb{Z}_{\ge N_2}$ with
\[
\log N_2 =\max\left\{\log N_1, e^{4e}, \sigma^{-2}, \beta^2,\left(\frac{\rho(2\beta +1)}{2\log(|s|/\alpha)} \right)^{4/3} \right\}\,.
\]
\end{corollary}

We note that there is nothing special about the constant $\rho$ that appears as a multiplier of the function $\varepsilon(N)$ above. The seemingly arbitrary choice is made just to keep $N_2$ explicit and relatively simple. For large enough $N$ any constant strictly larger than one suffices, as in fact
\[
\frac{z(\sigma \log N)\log\log N}{\sigma \log N}-1 \sim \varepsilon(N)\,.
\]
We will formalize this observation as a corollary of an asymptotic nature.
For more precision, we also investigate approximating $z(y)$ by $z_n(y)$, $n \in \mathbb{Z}_{\ge 2}$.

\begin{corollary}\label{ASYMPTOTICTHEOREM}
Let $s \in \mathbb{Z}\setminus\{0\}, t\in\mathbb{Z}_{\ge 1}$ and $gcd(s,t)=1$. \\
For any $\varepsilon_3 > 0$ there exists $N_3 \in \mathbb{Z}_{\ge 1}$ such that
\[
 1<\left|e^{s/t}-\frac{M}{N}\right|\,
N^{2+2\log(|s|/\alpha)\,(1+(1+\varepsilon_3)\varepsilon(N))/\log\log N}
\]
for all 
$M\in\mathbb{Z}\setminus\{0\}, N \in \mathbb{Z}_{\ge N_3}$.\\
For any $n \in \mathbb{Z}_{\ge 2}$ there exists $c_3 > 0$ and $N_3 \in \mathbb{Z}_{\ge 1}$ such that
\[
 1<\left|e^{s/t}-\frac{M}{N}\right|\,
N^{2+2\log(|s|/\alpha)(1/\log z_{n-1}(\log N)+c_3/(\log\log N)^2)}
\]
for all 
$M\in\mathbb{Z}\setminus\{0\}, N \in \mathbb{Z}_{\ge N_3}$.
\end{corollary}

If $|s|=3$, then we get quite dramatic improvement to the earlier results.
As an example we consider number $e^3$. A straightforward application of Corollary \ref{SPECIALTHEOREM} would give us 
\[
1<\left|e^{3}-\frac{M}{N}\right| 1629 \,
N^{2+\log 3\,(1+8\varepsilon(N))/\log\log N}\left(\frac{\log N}{\log\log N}\right)^{3}
\]
for all $M \in\mathbb{Z}\setminus\{0\}, N \in \mathbb{Z}_{\ge 1}$ with
$\log N \geq e^{4e}\approx 52740$. However by following the proof of Corollary \ref{SPECIALTHEOREM} and using sharper bounds we obtain a better result:
\begin{corollary}\label{EKOLOME}
Let $s=3$ and $t=1$. 
Then
\begin{equation}\label{Onthebest}
1<\left|e^{3}-\frac{M}{N}\right| 1561 \,
N^{2+\log 3\,(1+4\varepsilon(N))/\log\log N}\left(\frac{\log N}{\log\log N}\right)^{3}
\end{equation}
for all $M \in\mathbb{Z}\setminus\{0\}, N \in \mathbb{Z}_{\ge 1}$ with
$\log N \geq 983$.
\end{corollary}
If we compare the result \eqref{Onthebest} to the earlier results we see that 
on the best the earlier investigations give a term $2\log 3$ instead of $\log 3$ in the exponent of \eqref{Onthebest}
(see the next chapter for a general analysis).

\section{Comparison to existing results}

The simple continued fraction expansions of $e^{\pm 1/t}$ and $e^{\pm 2/t}$ are well known, see e.g. \cite{mccabe}. 
They are ideal for constructing lower bounds for $|e^{\pm 1/t}-M/N|$ and $|
e^{\pm 2/t} -M/N|$, which is done by
Davis \cite{davis}, \cite{davis1979}, Dodulikov\'a et al. \cite{hanclalmostperiodic} and Tasoev \cite{tasoev} for example.
In fact, 
all the convergents of \eqref{GENCONTFRACEXP} evaluated at $z = 1/t$ and $z = 2/t$ are also convergents of the simple continued fraction 
expansions of $e^{1/t}$ and $e^{2/t}$, respectively \cite{mccabe} 
(they cover asymptotically $1/3$ and $3/5$ of all the simple continued fraction convergents, respectively). 
However, our approach is intended for the general case, using the triangle inequality instead of approximation properties of the simple 
continued fractions, and thus we cannot match existing results when $|s| \le 2$.

When $|s| \ge 3$, the simple continued fraction expansions of $e^{s/t}$ are unavailable.
We then have $1 < \alpha < |s|$, which means that the common factor we found in the numerators and the denominators 
of our approximations is enough to give us an advantage over existing results.
Note also, that before extracting common factors we have $\alpha=1,\ \beta=0,\ \gamma=1$ in our Theorem \ref{GENERALTHEOREM}.
This will be a starting  platform for presenting already existing results.

Bundschuh \cite{bundschuh1979} used Pad\'e approximation technique to prove that
\[
1 \le \left| e^{s/t}-\frac{M}{N} \right|c_BN^{2+4(\log|s|)\,w(\log N)/\log N}\,w(\log N),
\quad w(x)=\frac{x+\delta}{\log (x+\delta)}
\]
for all $M\in\mathbb{Z}\setminus\{0\}, N \in \mathbb{Z}_{\ge 1}$,
where $c_B>0$ and $\delta>0$ were explicitly given. Bundschuh's result applies to any $N\ge 1$, 
while Corollary \ref{SPECIALTHEOREM} for instance requires $N$ to be rather big. 
On the other hand, Corollary \ref{SPECIALTHEOREM} has better asymptotic behaviour because of having the constant 
$2\log(|s|/\alpha)$ instead of $4\log|s|$ in the exponent.

Shiokawa \cite{shiokawa} used the representation \eqref{GENCONTFRACEXP} to show that
there exists a constant $c_S > 0$ such that
\begin{equation}\label{SHIOKAWA1}
1 \le \left| e^{s/t}-\frac{M}{N} \right| c_S N^{2+2(\log|s|) z(\tau \log N)/(\tau \log N)}\frac{\log N}{\log\log N}\,,
\end{equation}
where $\tau = 4t/(es^2)$,
for $M\in\mathbb{Z}\setminus\{0\}, N \in \mathbb{Z}_{\ge 3}$ 
(we modified the statement of Shiokawa to use our notation $z(y)$ and to apply also for negative $s$).
Note also that before extracting common factors (meaning $\alpha=1,\ \beta=0,\ \gamma=1$) our Theorem \ref{GENERALTHEOREM}
gives Shiokawa's result \eqref{SHIOKAWA1} with a completely explicit constant in place of $c_S$.
Further, if $|s|> 1$, then $\alpha>1$ and $\tau < \sigma$,
giving an improvement to \eqref{SHIOKAWA1} despite $\beta$ being positive.

Using Pad\'e approximants just below the diagonal, Zheng \cite{zheng} proved that for any 
$\varepsilon_Z > 0$ there exists $N_Z \in \mathbb{Z}_{\ge 1}$  such that
\begin{equation}\label{ZHENG}
1 \le \left|e^{s/t} - \frac{M}{N} \right|\frac{8te^{2|s|/t}}{s^2} N^{2+(2+\varepsilon_Z)(\log|s|)/\log\log N} \frac{\log N}{\log\log N}
\end{equation}
for $M\in\mathbb{Z}\setminus\{0\}, N \in \mathbb{Z}_{\ge N_Z}$.
As the relation between $\varepsilon_Z$ and $N_Z$ is not explicitly given, 
the statement \eqref{ZHENG}  is in fact equivalent after omitting the factors $8te^{2|s|/t}/s^2$ and $\log N / \log\log N$.
Again, when $\alpha > 1$ and because of the sharper treatment of $\varepsilon_3$, 
our Corollary \ref{ASYMPTOTICTHEOREM} is asymptotically better than \eqref{ZHENG} already with $n=1$.

Other effective results on the rational approximation problem include e.g. Mahler \cite{mahler} and Nesterenko et al. \cite{neswal}, 
neither of which has general bounds with the same precision as we do.

\section{Some facts about generalized continued fractions} 

By a generalized continued fraction we mean the expression
\begin{equation}\label{GC}
b_0+\frac{a_1}{b_1+\frac{a_2}{b_2+\cdots}}\,,
\end{equation}
for which we use standard notation
\[
b_0+\K_{n=1}^{\infty}\frac{a_n}{b_n}=
b_0+\frac{a_1}{b_1} \alaplus \frac{a_2}{b_2} \alaplus \aladots \,.
\]
The convergents of \eqref{GC} are defined by
\[
\frac{A_n}{B_n} = b_0+ \K_{k=1}^n \frac{a_k}{b_k}=b_0+\frac{a_1}{b_1}\alaplus \frac{a_2}{b_2}
\alaplus \aladots \alaplus \frac{a_n}{b_n} \,,
\]
where the numerators $A_n$ and denominators $B_n$ both satisfy the recurrence formula
\begin{equation}\label{ABRECURRENCES}
C_n = b_n C_{n-1} + a_n C_{n-2}
\end{equation}
for $n \in \mathbb{Z}_{\ge 2}$
with initial values $A_0=b_0$, $B_0=1$, $A_1=b_0b_1+a_1$ and $B_1=b_1$. 
By the value of \eqref{GC} we mean the limit
\[
\tau = \lim_{n\to\infty} \frac{A_n}{B_n} \,,
\]
when it exists.
Using recurrence formula \eqref{ABRECURRENCES} and induction gives
\[
\frac{A_{k+1}}{B_{k+1}}-\frac{A_k}{B_k}=
\frac{(-1)^k a_1a_2\cdots a_{k+1}}{B_kB_{k+1}}
\]
for all $k \in \mathbb{Z}_{\ge 0}$,
which by telescoping implies
\[
\frac{A_n}{B_n}=b_0+\sum_{k=0}^{n-1}\frac{(-1)^ka_1a_2\cdots a_{k+1}}{B_kB_{k+1}}\,.
\]
Supposing the convergence the limit
\[
\tau=b_0+\sum_{k=0}^{\infty}\frac{(-1)^ka_1a_2\cdots a_{k+1}}{B_kB_{k+1}}
\]
exists, which further implies
\[
\tau-\frac{A_n}{B_n}=\sum_{k=n}^{\infty}\frac{(-1)^ka_1a_2\cdots a_{k+1}}{B_kB_{k+1}}\,.
\]
This we may use for example in the following way.
If $a_k$ and $b_k$ are positive for $k \in \mathbb{Z}_{\ge 1}$, then we can use the standard error estimate of an alternating sum:
\begin{equation}\label{REMAINDER}
\left|\tau- \frac{A_n}{B_n}\right|<\frac{a_1a_2\cdots a_{n+1}}{B_{n}B_{n+1}} \quad\text{or}\quad
\left|B_{n}\tau- A_{n}\right|<\frac{a_1a_2\cdots a_{n+1}}{B_{n+1}}\,.\\
\end{equation}

\section{Proofs} 

The exponential function 
$e^{z}$ is expressible as a generalized continued fraction
\[
e^{z}=1+\frac{2z}{2-z} \alaplus \frac{z^2}{6} \alaplus \frac{z^2}{10} \alaplus \frac{z^2}{14}
\alaplus\aladots 
= 1+\frac{2z}{2-z+\K_{n=1}^{\infty} \frac{z^2}{4n+2}}
\]
for all $z\in\mathbb{C}$. 
When $z=s/t\in\Q$, $s\in\mathbb{Z}\setminus\{0\}$ and $t\in\mathbb{Z}_{\ge 1}$, we get
\begin{equation}\label{EXPKETJU}
e^{s/t}=1+\frac{2s}{2t-s} \alaplus \frac{s^2}{6t} \alaplus \frac{s^2}{10t}
\alaplus \frac{s^2}{14t} \alaplus \aladots 
=1+\frac{2s}{2t-s+E(s,t)}\,,
\end{equation}
where
\[
E(s,t)=\K_{n=1}^{\infty}\frac{a_n}{b_n}\,,\qquad a_n=s^2\,,\quad b_n=2t(2n+1) \,.
\]
The recurrence formula \eqref{ABRECURRENCES} for numerators $A_n$ and denominators $B_n$ of the convergents of $E(s,t)$ looks like
\begin{equation}\label{Cnrecurrence}
C_n = 2t(2n+1)C_{n-1}+s^2C_{n-2}
\end{equation}
for $n \in \mathbb{Z}_{\ge 2}$
with initial values $A_0=0$, $B_0=1$, $A_1=s^2$ and $B_1=6t$. Our proof for the main theorem
also uses two auxiliary sequences $(C^\pm_n)_{n \ge 0}$ defined by
\[
C_n^{\pm} = A_n+(2t\pm s)B_n\,.
\]
Clearly the numbers $C^\pm_n$, $n \ge 2$ satisfy the recurrence formula \eqref{Cnrecurrence} too
with initial values $C^\pm_0 = 2t\pm s$ and $C^\pm_1 = 12t^2\pm6ts+s^2 = 6tC^\pm_0+s^2$.

For a prime number $p$ and a natural number $n$ define the $p$-adic order of $n$ in the usual way:
\[
v_p(n) =
\begin{cases}
\max\{v \in \mathbb{Z}_{\ge 0} \colon p^v \mid n \}\,, \quad & n \ge 1\,, \\
\infty\,, \quad & n = 0\,.
\end{cases}
\] 
The function $v_p$ has the properties
\[
v_p(mn) = v_p(m) + v_p(n)
\]
and
\begin{equation} \label{nonarch}
v_p(m+n)
\begin{cases}
\ge \min\{v_p(m), v_p(n)\}\,, \quad & v_p(m) = v_p(n)\,, \\
= \min\{v_p(m), v_p(n)\}\,, \quad & v_p(m) \neq v_p(n)\,,
\end{cases}
\end{equation}
the non-Archimedean triangle inequality. We will need Legendre's theorem on $p$-adic orders of factorials:
\begin{theorem}[Legendre \cite{legendre}]\label{legendre}
Let $n = a_0+a_1p+\cdots+a_m p^m$, where $a_k \in \{0, 1, \ldots, p-1\}$ for $k=0,1,\dots,m$ and $a_m\ge 1$. Then
\[
v_p(n!) = \frac{n-s_p(n)}{p-1}\,,
\]
where
\[
s_p(n) = a_0+a_1+\cdots + a_m.
\]
\end{theorem}

Before the proof of our main result let us give two Lemmas on the greatest common divisor
of $C_n^+$ and $C_n^-$.

\begin{lemma}\label{ISOTEKIJA}
\[
D_n = \prod_{\substack{p \in \mathbb{P} \\ v_p(s) \ge 1}} p^{r_p(n)} \,\bigg|\, \gcd(C^+_n, C^-_n) \,,
\]
where
\[
\begin{cases}
r_2(n) = n+1\,, &\\
r_p(n) = v_p((n+1)!) \,, \quad & p \ge 3\,.
\end{cases}
\]
\end{lemma}
\begin{proof}
We show by induction that
\begin{equation}\label{explicit}
C^\pm_n = \sum_{k=0}^{n+1} \frac{(n+1+k)!}{k!(n+1-k)!}t^k(\pm s)^{n+1-k}\,.\nonumber
\end{equation}
First of all, with $n=0$ and $n=1$ we end up with the required initial values $C^\pm_0 = 2t\pm s$
and $C^\pm_1 = 12t^2 \pm 6ts+s^2$. For $n \ge 2$ we use the recurrence formula \eqref{Cnrecurrence}:
\begin{align*}
C^\pm_n & = 2t(2n+1)C^\pm_{n-1}+s^2C^\pm_{n-2} \\ 
&= 2t(2n+1)\sum_{k=0}^n \frac{(n+k)!}{k!(n-k)!}t^k(\pm s)^{n-k} + s^2 \sum_{k=0}^{n-1}\frac{(n-1+k)!}{k!(n-1-k)!}t^k(\pm s)^{n-1-k} \\
& =  \sum_{k=1}^{n+1} \frac{2(2n+1)(n-1+k)!}{(k-1)!(n+1-k)!}t^k (\pm s)^{n+1-k}
+ \sum_{k=0}^{n-1} \frac{(n-1+k)!}{k!(n-1-k)!}t^k(\pm s)^{n+1-k} \\
& = \frac{2(2n+1)(2n)!}{n!0!}t^{n+1}(\pm s)^0 + \frac{2(2n+1)(2n-1)!}{(n-1)!1!}t^n(\pm s)^1 \\
& \qquad \qquad + \sum_{k=1}^{n-1} \left(\frac{2(2n+1)(n-1+k)!}{(k-1)!(n+1-k)!}+\frac{(n-1+k)!}{k!(n-1-k)!} \right)
t^k (\pm s)^{n+1-k} \\
& \qquad \qquad \qquad \qquad + \frac{(n-1)!}{0!(n-1)!}t^0(\pm s)^{n+1} \\
& = \sum_{k=0}^{n+1} \frac{(n+1+k)!}{k!(n+1-k)!}t^k (\pm s)^{n+1-k}\,.
\end{align*}
Let $p$ be an odd prime dividing $s$, and thus not dividing $t$.
Using the non-Archimedean triangle inequality \eqref{nonarch} gives
\begin{align*}
v_p(C^\pm_n) & \ge \min_{0 \le k \le n+1} \left\{ v_p\left(\frac{(n+1+k)!}{k!(n+1-k)!}t^k(\pm s)^{n+1-k}\right)\right\} \\
&= \min_{0 \le k \le n+1} \left\{ v_p\left((n+1)!\binom{n+1+k}{k}t^k\frac{(\pm s)^{n+1-k}}{(n+1-k)!}\right)\right\} \\
&= \min_{0 \le k \le n+1} \bigg\{ v_p((n+1)!)+v_p\left(\binom{n+1+k}{k}\right)
\\ & \qquad \qquad \qquad \qquad \qquad \qquad +v_p(t^k)+v_p\left(\frac{(\pm s)^{n+1-k}}{(n+1-k)!}\right)\bigg\}\\
& \ge v_p((n+1)!) + \min_{0 \le k \le n+1} \left\{ (n+1-k)\left(v_p(s)-\frac{1}{p-1}\right)\right\}
\\ & = v_p((n+1)!)\,,
\end{align*}
where in the penultimate step we used Legendre's theorem \ref{legendre} to estimate 
\[
v_p((n+1-k)!)\le\frac{n+1-k}{p-1}\,.
\]

If $v_2(s) \ge 1$, then \[v_2(C^\pm_0)=v_2(2t\pm s)\geq 1\] and \[v_2(C^\pm_1)=v_2(6tC^\pm_0+s^2)\geq 2\,.\] If we assume that $v_2(C_k^\pm)\geq k+1$ for all $k<n$, then
\begin{align*}
v_2(C^\pm_n)&=v_2(2t(2n+1)C^\pm_{n-1}+s^2C^\pm_{n-2})
\\ & \ge \min\{v_2(2t(2n+1)C^\pm_{n-1}),v_2(s^2C^\pm_{n-2})\}\\
&=\min\{v_2(C^\pm_{n-1}) + 1, v_2(C^\pm_{n-2}) + 2v_2(s) \}
\\ &\ge \min\{n+ 1, n-1 + 2 \}=n+1\,.
\end{align*}
Hence by induction $v_2(C^\pm_n)\ge n+1$ for all $n\ge 0$.
\end{proof}

\begin{lemma}
\[
\frac{\alpha^{n+1}}{\gamma(n+1)^\beta}\le D_n \le \gcd(2,s)\alpha^n \,.
\]
\end{lemma}
\begin{proof}
If $v_2(s) \ge 1$ then $2$ contributes to $D_n$ by $2^{n+1}$ as shown in Lemma \ref{ISOTEKIJA}.

Let $p$ be an odd prime dividing $s$. Let $n+1 = a_0+a_1p+\cdots+a_m p^m$, $a_k \in \{0,1,\ldots, p-1\}$ for $k=0,1,\ldots,m$ and $a_m \ge 1$.
Estimating 
\[
m \le \frac{\log(n+1)}{\log p}
\]
and $1 \le s_p(n+1) \le (m+1)(p-1)$ in Legendre's theorem \ref{legendre} yields
\[
\frac{n+1}{p-1}-\frac{\log(n+1)}{\log p} -1 \le v_p((n+1)!) \le \frac{n}{p-1}\,.
\]
The claim follows.
\end{proof}

\begin{proof}[Proof of Theorem~{\rm\ref{GENERALTHEOREM}}]
From the convergents of $E(s,t)$ we get a sequence of linear forms 
\begin{equation}\label{SIMULTA1} 
R_n = B_n E(s,t)-A_n\,. 
\end{equation}
First we note that the remainder $R_n$ may be estimated by property \eqref{REMAINDER}, namely
\[
|R_n|<\frac{a_1a_2\cdots a_{n+1}}{B_{n+1}}\,.
\]
By \eqref{EXPKETJU} we have
\[
E(s,t)=\frac{2s}{e^{s/t}-1}+s-2t\,,
\]
which substituted into \eqref{SIMULTA1} implies
\[
(1-e^{s/t})R_n=C_n^{-}e^{s/t}-C_n^{+}\,, \qquad C_n^{\pm} = A_n+(2t\pm s)B_n\,.
\]
Lemma \ref{ISOTEKIJA} tells that
\[
\frac{C_n^{\pm}}{D_n}\in\mathbb{Z}\,,\quad n\in\mathbb{Z}_{\ge 0}\,.
\]
By denoting
\[
L_n =\frac{(1-e^{s/t})R_n}{D_n}\,,\qquad J_n=\frac{C_n^{+}}{D_n}\,,\qquad H_n=\frac{C_n^{-}}{D_n}\,,
\]
we get new linear forms
\[
L_n = H_ne^{s/t}-J_n\,, \qquad J_n\,,\, H_n\in\mathbb{Z}\,.
\]

Fix now $M\in\mathbb{Z}$ and $N\in\mathbb{Z}_{\ge 1}$.
Our target is to give a lower bound for the linear form 
\[
\Lambda = Ne^{s/t}-M \,.
\]
Denote also
\[
\Omega_n = NJ_n - MH_n = H_n\Lambda - NL_n\in\mathbb{Z}\,.
\]
If $\Omega_n \ne 0$, then by triangle inequality we get
\begin{equation}\label{KIIKKU}
1 \le |\Omega_n|\le |H_n||\Lambda|+N|L_n|\,.
\end{equation}
Since $L_n \rightarrow 0$ and $1\leq 2N|L_0|$ by \eqref{RAJAN},
there exists the biggest $\hat n\in\mathbb{Z}_{\ge 0}$ such that
\begin{equation}\label{fixhatn}
\frac{1}{2}\le N\left|L_{\hat n}\right|\,.
\end{equation}
By definition
\begin{equation}\label{fixhatn2}
N|L_k|<\frac{1}{2}
\end{equation}
for all $k\ge\hat n+1$. 
Since 
\[
J_nH_{n+1}-J_{n+1}H_n= \frac{2s}{D_nD_{n+1}}\left(A_{n+1}B_n-A_nB_{n+1}\right)=
\frac{(-1)^n2s^{2n+3}}{D_nD_{n+1}}\neq0
\]
for all $n\in\mathbb{Z}_{\ge 0}$, we see that $\Omega_{\hat n+1}\ne 0$ or $\Omega_{\hat n+2}\ne 0$. 
Hence we get from \eqref{KIIKKU} and \eqref{fixhatn2} the estimate
\begin{equation}\label{pihvi}
1<2|\Lambda||H_{\hat n+j}|\, ,
\end{equation}
where $j=1$ or $j=2$. Now we need to find an upper bound  for $|H_{\hat n+j}|$ depending on $N$. 
By noting that
\begin{align*}
H_{\hat n+j}&=\frac{C_{\hat n+j}^{-}}{D_{\hat n+j}}=
\frac{A_{\hat n+j}+(2t-s)B_{\hat n+j}}{D_{\hat n+j}}=
\frac{E(s,t)B_{\hat n+j}-R_{\hat n+j}+(2t-s)B_{\hat n+j}}{D_{\hat n+j}}\\
&=\frac{1}{D_{\hat n+j}}\left(\frac{2s}{e^{s/t}-1}B_{\hat n+j}-R_{\hat n+j}\right)\\ & =
\frac{1}{D_{\hat n+j}}\left(\frac{2s}{e^{s/t}-1}
\left(b_{\hat n+j}B_{\hat n+j-1}+a_{\hat n+j}B_{\hat n+j-2}\right)-R_{\hat n+j}\right)
\end{align*}
we are led to the estimate
\begin{align}\label{LEDESTIMATE}
|H_{\hat n+j}|&\le
\frac{1}{D_{\hat n+j}}\left(\frac{2|s|}{|e^{s/t}-1|}
(b_{\hat n+j}+a_{\hat n+j})B_{\hat n+j-1}+|R_{\hat n+j}|\right) \nonumber \\ 
&\le\frac{1}{D_{\hat n+1}}\left(\frac{2|s|}{|e^{s/t}-1|}
(b_{\hat n+2}+a_{\hat n+2})B_{\hat n+1}+|R_{\hat n+j}|\right).
\end{align}
Next we turn to estimation of the terms $B_{\hat n+1}$ and $|R_{\hat n+j}|$.
By \eqref{fixhatn} and \eqref{REMAINDER} we get 
\[
\frac{1}{2N}\le |L_{\hat n}|=
\left|\frac{(1-e^{s/t})R_{\hat n}}{D_{\hat n}}\right|<
\frac{|e^{s/t}-1|s^{2(\hat n+1)}}{D_{\hat n}B_{\hat n+1}}\,,
\]
and so we have
\begin{equation}\label{BHATTU2}
B_{\hat n+1}\le  
\frac{2N|e^{s/t}-1|s^{2(\hat n+1)}}{D_{\hat n}} \,.
\end{equation}
Further, by \eqref{fixhatn2} we have
\begin{equation}\label{RHATTU2}
|R_{\hat n+j}|= \left| \frac{D_{\hat n+j}L_{\hat n+j}}{1-e^{s/t}}\right|
<\frac{D_{\hat n+2}}{2N|e^{s/t}-1|}\,.
\end{equation}
Thus, combining \eqref{LEDESTIMATE} with \eqref{BHATTU2} and \eqref{RHATTU2} yields to
\begin{align}
|H_{\hat n+j}| & \le
\frac{1}{D_{\hat n+1}}\bigg(\frac{2|s|}{|e^{s/t}-1|}
(2t(2{\hat n}+5)+s^2)\frac{2N|e^{s/t}-1|s^{2(\hat n + 1)}}{D_{\hat n}}
\\ & \qquad \qquad\qquad\qquad\qquad\qquad\qquad\qquad\qquad\qquad+\frac{D_{\hat n+2}}{2N|e^{s/t}-1|}\bigg) \notag \\
& =
\frac{4Ns^{2(\hat n+1)}|s|(4t(\hat n + 1)+6t+s^2)}{D_{\hat n}D_{\hat n+1}}
+\frac{D_{\hat n+2}}{2N D_{\hat n+1} |e^{s/t}-1|} \,.\label{LEDESTIMATE2}
\end{align}

On the other hand, by the recurrence relation \eqref{ABRECURRENCES}, 
Stirling's formula 
\[
\Gamma(x+1) = \sqrt{2\pi}x^{x+1/2}e^{-x+\theta(x)/(12x)}\,, \qquad  0<\theta(x)<1\,,\quad x > 0\,,
\]
(see e.g. \cite{stirling}) and the fact that $\Gamma(3/2)=\sqrt{\pi}/2$ 
we deduce the lower bound
\begin{align}
B_{\hat n+1}&=
b_{\hat n+1}B_{\hat n}+a_{\hat n+1}B_{\hat n-1}
>b_{\hat n+1}B_{\hat n}>\cdots>
\prod_{k=1}^{\hat n+1}b_k=\prod_{k=1}^{\hat n+1} 2t(2k+1) \notag \\
&=(4t)^{\hat n+1}\prod_{k=1}^{\hat n+1}(k+\frac{1}{2})=(4t)^{\hat n+1}\frac{\Gamma(\hat n+\frac{5}{2})}{\Gamma(\frac{3}{2})}  \\ &> 2\sqrt{2}(\hat n+\tfrac{3}{2})^{\hat n+2}(4t)^{\hat n + 1}
e^{-(\hat n+3/2)}\,. \label{BHATTU3}
\end{align}
The bounds \eqref{BHATTU2} and \eqref{BHATTU3} together imply
\begin{align*}
\log N
&> \log(2\sqrt{2}) + (\hat n + 2)\log(\hat n+ \tfrac{3}{2}) + (\hat n + 1)\log(4t) -(\hat n + \tfrac{3}{2}) \\
& \qquad -\log 2 -\log|e^{s/t}-1|-2(\hat n + 1)\log s+ \log D_{\hat n} \\
&> \log(2\sqrt{2}) + (\hat n + 2)\log(\hat n+\tfrac{3}{2}) + (\hat n + 1)\log(4t) -(\hat n + \tfrac{3}{2}) \\
& \qquad -\log 2 -\log|e^{s/t}-1|-2(\hat n + 1)\log s\notag \\ & \qquad+ (\hat n+1)\log \alpha -\beta\log(\hat n+1)-\log \gamma\,.
\end{align*}
In order to get rid of the non-significant terms we assume $\hat n >\beta-1$ and make technical estimates $\log(\hat n + 3/2) > \log(\hat n + 1 - \beta)$ and $\log(\hat n + 3/2) > \log(\hat n + 1)$ in a specific proportion and continue
\begin{align}
\log N & > (\hat n +1 - \beta) \left(\log (\hat n + 1 - \beta)+\log(4t)-1-\log(s^2)+\log \alpha  \right) \notag \\
& \qquad + \log(2\sqrt{2}) + \log(\hat n+ \tfrac{3}{2}) + \beta\log(4t)-\beta -\tfrac{1}{2}-\log 2  \notag \\
& \qquad-\log|e^{s/t}-1|-\beta\log{s^2} +\beta\log \alpha -\log \gamma \notag \\
&  = (\hat n+1-\beta) \log \left( \frac{4t\alpha(\hat n+1-\beta)}{es^2} \right)
+ \log \left( \frac{\sqrt2 (4t\alpha)^\beta (\hat n + \frac{3}{2})}{e^{\beta +1/2}|e^{s/t}-1| s^{2\beta}\gamma } \right)
\label{twoterms} \,.
\end{align}
Suppose now that
\begin{equation}\label{hatnlus1p}
\hat n +1\ge \frac{e^{\beta +1/2}|e^{s/t}-1| s^{2\beta}\gamma}{\sqrt2 (4t\alpha)^\beta} - \frac{1}{2}\,.
\end{equation}
Then the second term in \eqref{twoterms} is non-negative and with the notation
\[
\sigma = \frac{4t\alpha}{es^2}
\]
we have
\[
\sigma \log N > \sigma (\hat n +1 - \beta)\log(\sigma (\hat n + 1 - \beta))\,.
\]
Suppose further that
\begin{equation}\label{invehto}
\sigma(\hat n+1-\beta) \ge e,
\end{equation} 
thereby we can use the inverse function $z(y)$ 
of the function $y(z)=z\log z$ and get
\[
\sigma (\hat n+1-\beta) < z(\sigma \log N)\,.
\]
Then
\begin{equation}\label{HATUNYLARAJA}
\hat n+1<\frac{z(\sigma\log N)}{\sigma}+\beta.
\end{equation}

Our assumptions $\hat n>\beta-1$, \eqref{hatnlus1p} and \eqref{invehto} are valid when
\begin{equation}\label{hatnlusplusp}
\hat n+1\ge \eta=\max\left\{\frac{\sqrt{e}|e^{s/t}-1|\gamma }{\sqrt2\sigma^{\beta}} - \frac{1}{2},\frac{e}{\sigma}+\beta\right\}\,.\nonumber
\end{equation}
If $\hat n+1<\eta$ then by the assumption \eqref{RAJAN} the upper bound \eqref{HATUNYLARAJA} holds anyway. Hence by the estimate \eqref{LEDESTIMATE2} we get
\begin{align*}
|H_{\hat n+j}|&\le
\frac{4Ns^{2(\hat n+1)}|s|(4t(\hat n + 1)+6t+s^2)}{D_{\hat n}D_{\hat n+1}}
+\frac{D_{\hat n+2}}{2N D_{\hat n+1} |e^{s/t}-1|} \\ 
& \le
N\left(\frac{s^2}{\alpha^2} \right)^{\hat n +1}\!\!
\bigg( 4\alpha^{-1}((\hat n +1)(\hat n + 2))^\beta \gamma^2 |s|(4t(\hat n + 1)+6t+s^2) \\
&\qquad\qquad 
+\left(\frac{\alpha^2}{s^2}\right)^{\hat n +1}\!\!\frac{\gcd(2,s)(\hat n+2)^\beta \gamma}{2N^2 |e^{s/t}-1|} \bigg)\\ 
& \le
N\!\left(\frac{s^2}{\alpha^2} \right)^{z(\sigma \log N)/\sigma+\beta}\!\!\! \bigg(
4\alpha^{-1}\!\left(\!\big(\frac{z(\sigma \log N)}{\sigma} +\beta\big)\!\big(\frac{z(\sigma \log N)}{\sigma} +\beta+1\big)\!\right)^{\!\beta}  \\
&\qquad \qquad \qquad \quad \times \gamma^2 |s|(4t\big(\frac{z(\sigma \log N)}{\sigma}  + \beta\big)+6t+s^2)
\\
&\qquad\qquad 
+\frac{\gcd(2,s)\alpha^2\gamma}{2N^2 s^2|e^{s/t}-1|}\big(\frac{z(\sigma \log N)}{\sigma} +\beta+1\big)^\beta\bigg)
\\
&\le \tfrac{1}{2} N^{1+2\log(|s|/\alpha)z(\sigma\log N)/(\sigma\log N)} Z(N)\,.
\end{align*}
The result then follows from \eqref{pihvi}.
\end{proof}

Before proving the corollaries we present some properties of the inverse function $z(y)$ of the function
$y(z)= z \log z$, $z \geq 1/e$.

\begin{lemma}\label{inverse}
The inverse function $z(y)$ of the function
$y(z)= z \log z$, $z \geq 1/e$,
is strictly increasing. 
Define $z_0(y)=y$ and $z_n(y)=y/\log z_{n-1}(y)$ for $n\in\mathbb{Z}_{\ge 1}$. 
If $y>e$, then
\[
z_1<z_3<\cdots <z<\cdots <z_2<z_0
\]
and the inverse function may be given by the infinite nested logarithm fraction
\[ 
z(y) = \underset{n\to\infty}{\lim} z_{n}(y)=\frac{y}{\log \frac{y}{\log \frac{y}{\log \cdots}}}\,. 
\]
For the error made when approximating $z(y)$ by $z_n(y)$ we have
\begin{equation} \label{zerror}
\begin{cases}
z(y) - z_0(y) & = y\left(\dfrac{1}{\log z(y)} - 1\right)\,, \\
z(y) - z_1(y) & = \dfrac{z_1(y) \log\log z(y)}{\log z(y)} \,, \\
|z(y) - z_n(y)| & \le \dfrac{ (\log y)^{\lfloor n/2 \rfloor} z_1(y) \log\log z(y)}{(\log z_1(y))^{n-1}(\log z(y))^{\lfloor n/2 \rfloor + 1}} \,, \qquad n \in \mathbb{Z}_{\ge 2} \,.
\end{cases}
\end{equation}
\end{lemma}
\begin{proof}
Everything but the errors \eqref{zerror} is proven in \cite{hancl1}.
The case $n=0$ is immediate, and the case $n=1$ is true as  
\begin{align*}
z(y)& =\frac{y}{\log z(y)}= \frac{y\log y}{\log y \log z(y) }
=\frac{y  \log (z(y)\log z(y))}{\log y \log z(y)}
 \\ & =\frac{y(\log z(y) + \log\log z(y))}{\log y\log z(y)} 
  =z_1(y)\left(1+\frac{\log\log z(y)}{\log z(y)}\right)\,.
\end{align*}
When $n \ge 2$ is odd, we can bound
\begin{align*}
0 & <  z(y) -z_n(y)  = \frac{y}{\log z(y)} - \frac{y}{\log z_{n-1}(y)}
= \frac{y \log (z_{n-1}(y) /z(y))}{\log z(y) \log z_{n-1}(y)} \\
& \le  \frac{y (z_{n-1}(y) /z(y) - 1)}{\log z(y) \log z_{n-1}(y)}
=  \frac{y (z_{n-1}(y) - z(y))}{z(y)\log z(y) \log z_{n-1}(y)} \\ &
=  \frac{z_{n-1}(y) - z(y)}{\log z_{n-1}(y)} \le \frac{z_{n-1}(y) - z(y)}{\log z_1(y)}\,.
\end{align*}
When $n \ge 2$ is even, we only get
\begin{align*}
0 & <  z_n(y)-z(y)   = \frac{y}{\log z_{n-1}(y)}-\frac{y}{\log z(y)}
= \frac{y \log (z(y) /z_{n-1}(y))}{\log z(y) \log z_{n-1}(y)} \\
& \le  \frac{y (z(y) /z_{n-1}(y) - 1)}{\log z(y) \log z_{n-1}(y)}
=  \frac{y (z(y) - z_{n-1}(y))}{z_{n-1}(y)\log z(y) \log z_{n-1}(y)} \\ &
\le \frac{y(z(y) - z_{n-1}(y))}{z_1(y)\log z(y)\log z_1(y)} = \frac{\log y(z(y) - z_{n-1}(y))}{\log z(y)\log z_1(y)}\,.
\end{align*}
The claim follows.
\end{proof}

\begin{proof}[Proof of Corollary~{\rm\ref{SPECIALTHEOREM}}]
Suppose $M \in \mathbb{Z}\setminus\{0\}$ and $N \in \mathbb{Z}_{\ge N_2}$.

If $\sigma \ge 1$, then
\[
\log z_1(\sigma \log N) \ge \log\log N - \log\log\log N\,.
\]
Since $\log N \ge e^{4e}$, which also implies
\begin{equation} \label{apuepis}
4\log\log\log N \le \log\log N\,,
\end{equation}
we now have
\begin{equation} \label{logz1pos}
\log z_1(\sigma \log N) \ge \tfrac{3}{4} \log\log N \,.
\end{equation}
Then by Lemma \ref{inverse} and the fact that $\log x /x$ is strictly decreasing when $x>e$ we get
\[
\frac{\log\log z(\sigma \log N)}{\log z(\sigma \log N)}\le
\frac{\log\log z_1(\sigma \log N)}{\log z_1(\sigma \log N)} \le
\frac{\log(\tfrac{3}{4}\log\log N)}{\tfrac{3}{4}\log \log N}
\le \tfrac{4}{3} \varepsilon(N)\,.
\]
Hence by Lemma \ref{inverse} we have
\begin{align}
z(\sigma \log N) &=
\frac{\sigma \log N}{\log \left( \sigma \log N\right)} \left( 1+ \frac{\log\log z(\sigma \log N)}{\log z(\sigma \log N)} \right) \notag 
\\ &  \le \frac{\sigma \log N}{\log \log N}\left( 1+ \tfrac{4}{3} \varepsilon(N)\right)\,.
\label{ztoz1pos}
\end{align}

Then suppose that $\sigma < 1$.
The assumption $\log N \ge \sigma^{-2}$ is equivalent with inequality $\log(\sigma \log N) \ge \log\log N/2$.
Using \eqref{apuepis} we can estimate
\begin{equation} \label{logz1}
\log z_1(\sigma \log N) \ge \tfrac{1}{2}\log\log N - \log\log\log N \ge \tfrac{1}{4} \log\log N \,,
\end{equation}
which is also implied by \eqref{logz1pos} when $\sigma \ge 1$.
As before we get
\begin{equation}\label{epsilon}
\frac{\log\log z(\sigma \log N)}{\log z(\sigma \log N)}\le
\frac{\log\log z_1(\sigma \log N)}{\log z_1(\sigma \log N)} \le
\frac{\log(\tfrac{1}{4}\log\log N)}{\tfrac{1}{4}\log \log N}
\le 4 \varepsilon(N)\,.
\end{equation}
The assumption $\log N\ge e^{4e}$ implies $\log\log\log N \ge 2$ and the assumption $\log N \ge \sigma^{-2}$ implies $-\log \sigma/\log\log N \le 1/2$.
By Lemma \ref{inverse} and inequality \eqref{epsilon}
\begin{align}
z(\sigma \log N) &=
\frac{\sigma \log N}{\log \left( \sigma \log N\right)} \left( 1+ \frac{\log\log z(\sigma \log N)}{\log z(\sigma \log N)} \right) \notag 
\\ &  \le
\frac{\sigma \log N}{\log\log N} \left( 1+ 4\varepsilon(N) \right) \left( 1+\frac{-\log \sigma}{\log\log N}+\left( \frac{-\log \sigma}{\log\log N}\right)^2 + \cdots \right)
\notag \\ & \le
\frac{\sigma \log N}{\log\log N} \left( 1+ 4\varepsilon(N) \right)\left( 1-\frac{2\log \sigma}{\log\log N}\right)\notag \\ & =
\frac{\sigma \log N}{\log\log N} \left( 1- \frac{2\log \sigma}{\log\log N}+4\varepsilon(N) -\frac{8\varepsilon(N)\log \sigma}{\log\log N}\right) \notag \\ & \le
\frac{\sigma \log N}{\log\log N} \left( 1+ (4-2\log \sigma)\varepsilon(N) \right)\,. \label{ztoz1neg}
\end{align}

Altogether \eqref{ztoz1pos} and \eqref{ztoz1neg} can be summarized as
\begin{equation}\label{ztoz1}
z(\sigma \log N) \le \frac{\sigma \log N}{\log\log N} \left(1+(\rho-1) \epsilon(N)\right)\,.
\end{equation}

Because $\eta-\beta\ge e/\sigma$, we have
\[
\log N\geq  (\eta-\beta) \log\left(\sigma(\eta-\beta)\right)\ge \frac{e}{\sigma}\,.
\]
Thus
Lemma \ref{inverse} and inequality \eqref{logz1} yield
\begin{align*}
\zeta(N)& =\frac{z(\sigma \log N)}{\sigma}+\beta \ge \frac{z_1(\sigma \log N)}{\sigma}+\beta 
\ge \frac{(\log N)^{1/4}}{\sigma}+\beta \\ &\ge \frac{(\log N_2)^{1/4}}{\sigma}+\beta=\frac{1}{d}\,.
\end{align*}
Now
\begin{align}
Z(N)&  =
\frac{8\gamma^2 |s|(4t\zeta(N)+6t+s^2)}{\alpha}\left(\frac{s^2}{\alpha^2} \zeta(N)(\zeta(N)+1)\right)^\beta \notag \\ & \qquad\qquad\qquad\qquad
+\frac{\gcd(2,s)\alpha^2\gamma}{N^2 s^2|e^{s/t}-1|}\left(\frac{s^2}{\alpha^2}(\zeta(N)+1)\right)^\beta \notag \\
&\le \zeta(N)^{2\beta+1} \bigg(\frac{8\gamma^2 |s|^{2\beta +1}(1+d)^{\beta}\left(4t+d(6t+s^2)\right)}{\alpha^{2\beta + 1}} \notag \\ & \qquad\qquad\qquad\qquad\qquad\qquad
+\frac{\gcd(2,s)\gamma s^{2(\beta-1)}d(d+d^2)^{\beta}}{N_2^2\alpha^{2(\beta - 1)}|e^{s/t}-1|}\bigg)
\notag \\ &
 = c_2 \zeta(N)^{2\beta+1} \,. \label{Z}
\end{align}
Inequality \eqref{apuepis} implies 
\begin{equation}\label{valivaihe}
(\log N)^{1/2}\log\log N\leq \frac{\log N}{\log\log N}=\varepsilon(N)\frac{\log N}{\log\log\log N}\, .
\end{equation}
Since $\log N\geq \beta^2$ and $\log\log\log N> 1$, \eqref{valivaihe} gives
$\beta \log\log N < \log N \varepsilon(N)$. By \eqref{ztoz1} we have
\begin{align}
\zeta(N) & = \frac{z(\sigma \log N)}{\sigma} +\beta < \frac{\log N}{\log\log N}
\left(1+(\rho - 1)\varepsilon(N) \right)+\beta \notag \\ & <
\frac{\log N}{\log\log N}
\left(1+\rho\varepsilon(N) \right)\,. \label{zeta}
\end{align}
Finally since 
\[
\log N \ge \left(\frac{\rho(2\beta +1)}{2\log(|s|/\alpha)} \right)^{4/3}
\]
and since \eqref{apuepis} implies $(\log N)^{3/4}<\log N/\log\log N$, we can estimate
\begin{align*}
(2\beta+1)\log(1+\rho\varepsilon(N))&<(2\beta+1)\rho\varepsilon(N)\\ &=(2\beta+1)\rho\varepsilon(N)(\log N)^{3/4}(\log N)^{-3/4}\\
&<(2\beta+1)\rho\varepsilon(N)\frac{\log N}{\log\log N}\left(\left(\frac{\rho(2\beta +1)}{2\log(|s|/\alpha)} \right)^{4/3}\right)^{-3/4}\\
&<2\log(|s|/\alpha)\varepsilon(N)\frac{\log N}{\log\log N}
\end{align*}
and get
\begin{equation}\label{epsiloncontrol}
\left(1+\rho\varepsilon(N)\right)^{2\beta +1} \le N^{2\log(|s|/\alpha)\varepsilon(N)/\log\log N}\,.
\end{equation}

Using Theorem \ref{GENERALTHEOREM} and inequalities \eqref{ztoz1}, \eqref{Z}, \eqref{zeta}
and then \eqref{epsiloncontrol} gives
\begin{align*}
1&<\left|e^{s/t}-\frac{M}{N}\right|Z(N)\,
N^{2+2\log(|s|/\alpha)z(\sigma\log N)/(\sigma\log N)}\\
&<\left|e^{s/t}-\frac{M}{N}\right| c_2 \,
N^{2+2\log(|s|/\alpha)(1+(\rho - 1)\varepsilon(N))/\log\log N} \left(\frac{\log N}{\log\log N}\left( 1+\rho\varepsilon(N)\right)\right)^{2\beta+1} \\ &
<\left|e^{s/t}-\frac{M}{N}\right| c_2 \,
N^{2+2\log(|s|/\alpha)(1+\rho\varepsilon(N))/\log\log N}\left(\frac{\log N}{\log\log N}\right)^{2\beta+1}\,,
\end{align*}
which concludes the proof.
\end{proof}

\begin{proof}[Proof of Corollary~{\rm\ref{ASYMPTOTICTHEOREM}}]
Let's show by induction that
\begin{equation}\label{sigmapois}
z_n(\sigma \log N) = \sigma z_n(\log N)\left(1+O\left(\frac{1}{\log\log N}\right)\right)\,.\nonumber
\end{equation}
With $n=0$ there is nothing to prove. For $n \ge 1$ we use Lemma \ref{inverse} to estimate
\[
\log z_{n-1}(\log N) > \log z_1(\log N) = \log\log N - \log\log\log N
\]
and 
to calculate
\begin{align*}
z_n(\sigma \log N) & = \frac{\sigma \log N}{\log z_{n-1}(\sigma \log N)} = \frac{\sigma \log N}{\log \left( \sigma z_{n-1}(\log N)\left(1+O\left(\frac{1}{\log\log N}\right) \right)\right)} \\
& = \frac{\sigma \log N}{\log\sigma +\log z_{n-1}(\log N) + O\left(\frac{1}{\log\log N} \right)}\\ & = \frac{\sigma \log N}{\log z_{n-1}(\log N) + O(1)} \\
& = \frac{\sigma \log N}{\log z_{n-1}(\log N)\left(1 + O\left(\frac{1}{\log z_{n-1}(\log N)}\right)\right)} \\ &= \frac{\sigma \log N}{\log z_{n-1}(\log N)\left(1 + O\left(\frac{1}{\log\log N}\right)\right)} \\
& = \frac{\sigma \log N}{\log z_{n-1}(\log N)}\left(1 + O\left(\frac{1}{\log\log N}\right)\right) \\ & =\sigma  z_n(\log N)\left(1 + O\left(\frac{1}{\log\log N}\right)\right) \,.
\end{align*}
In other words, if $n$ is even then there exist a constant $c_3 > 0$ and an integer $N_3 \in \mathbb{Z}_{\ge 1}$ such that
\begin{equation}\label{even}
\frac{z(\sigma \log N)}{\sigma \log N} \le 
\frac{z_n(\sigma \log N)}{\sigma \log N} \le\frac{1}{\log z_{n-1}(\log N)}+\frac{c_3}{(\log\log N)^2}
\end{equation}
for all $N \ge N_3$.

With odd $n$ we also need to consider the errors made in the approximation of $z(y)$ by $z_n(y)$.
For starters,
\[
\frac{\log\log z(\sigma \log N)}{\log z(\sigma \log N)} = \frac{\log\log\log N}{\log\log N}(1+o(1))\,.
\]
Because
\begin{align*}
\log (\sigma \log N)(1+o(1)) & = \log\log N \left(1+\frac{\log \sigma}{\log\log N}\right)(1+o(1))
\\ & = \log\log N (1+o(1))\,,
\end{align*}
\begin{align*}
z_1(\sigma \log N)(1+o(1)) & = \frac{\sigma \log N(1+o(1))}{ \log\log N \left(1+\frac{\log \sigma}{\log\log N}\right)} \\ & =
\frac{\sigma \log N(1+o(1))}{ \log\log N \left(1+O\left(\frac{1}{\log\log N}\right)\right)} 
\\ & = \frac{\sigma \log N(1+o(1))\left(1+O \left(\frac{1}{\log\log N}\right)\right)}{ \log\log N} \\ & =  \frac{\sigma \log N(1+o(1))}{ \log\log N}
\end{align*}
and
\begin{align*}
\frac{1+o(1)}{\log z_1(\sigma \log N)} & = \frac{1+o(1)}{\log \sigma + \log \log N - \log\log(\sigma \log N)} \\ & = \frac{1+o(1)}{\log\log N \left(1+O\left(\frac{\log\log\log N}{\log\log N}\right) \right)} \\ & = \frac{(1+o(1))\left(1+O\left(\frac{\log\log\log N}{\log\log N}\right) \right)}{\log\log N }\\ & = \frac{1+o(1)}{\log\log N }\,,
\end{align*}
we have
\begin{align*}
z(\sigma \log N)-z_n(\sigma \log N) & \le \frac{ (\log (\sigma \log N))^{\lfloor n/2 \rfloor} z_1(\sigma \log N) \log\log z(\sigma \log N)}{(\log z_1(\sigma \log N))^{\lfloor 3n/2\rfloor-1}\log z(\sigma \log N)}  \\
&  = \frac{\sigma \log N \log\log\log N (1+o(1))}{(\log\log N)^{n+1}} 
\end{align*}
by Lemma \ref{inverse}. So if $n$ is odd then for any constant $\varepsilon_3 > 0$ there exists a constant $c_3 > 0$ and an integer $N_3 \in \mathbb{Z}_ {\ge 1}$ such that
\begin{align}
\frac{z(\sigma \log N)}{\sigma \log N} & 
= \frac{z_n(\sigma \log N)}{\sigma \log N} +\frac{z(\sigma \log N)-z_n(\sigma \log N)}{\sigma \log N} \notag \\ & 
\le  \frac{1}{\log z_{n-1}(\log N)}+\frac{c_3}{(\log\log N)^2}+\frac{(1+\varepsilon_3)\log\log\log N}{(\log\log N)^{n+1}} \label{odd}
\end{align}
when $N \ge N_3$. Which of the two error terms is dominant depends on whether $n = 1$ or $n \ge 3$.

To conclude the proof we use inequalities \eqref{even} and \eqref{odd} to Theorem \ref{GENERALTHEOREM}, noting that because
\[
Z(N) = O\left(\zeta(N)^{2\beta + 1}\right) = O\left(\left(\frac{\log N}{\log\log N}\right)^{2\beta +1}\right)
\]
then
\[
\frac{\log Z(N)}{\log N} = O\left(\frac{\log\log N}{\log N}\right)
= O\left(\frac{1}{(\log\log N)^2}\right)\,,
\]
and hence the function $Z(N)$ is not significant.
\end{proof}

\begin{proof}[Proof of Corollary \ref{EKOLOME}]
Now we have the values $\alpha=\sqrt{3}$, $\beta=1$, $\gamma=3$, $\sigma=4\sqrt{3}/(9e)<1$ and $\log N_1= 982.40529\ldots<983$. Let $M \in \mathbb{Z}\setminus\{0\}$ and $N \in \mathbb{Z}_{\ge 1}$ with $\log N\geq 983$. 
Because $\log N\geq 983>\sigma^{-5}$, we can estimate
\[
\log z_1(\sigma \log N) \ge \tfrac{4}{5}\log\log N - \log\log\log N\,.
\]
Since $\log\log N \ge \log 983 >2e$, we have
\begin{equation} \label{apuepis2}
\varepsilon(N)=\frac{\log\log\log N}{\log\log N}<\frac{\log\log 983}{\log 983}<\frac{3}{10}
\end{equation}
and hence
\begin{equation} \label{logz12}
\log z_1(\sigma \log N) \ge \tfrac{1}{2} \log\log N >e\,.\nonumber
\end{equation}
Then
\begin{equation}\label{epsilon2}
\frac{\log\log z(\sigma \log N)}{\log z(\sigma \log N)}\le
\frac{\log\log z_1(\sigma \log N)}{\log z_1(\sigma \log N)} \le
\frac{\log(\tfrac{1}{2}\log\log N)}{\tfrac{1}{2}\log \log N}
\le 2 \varepsilon(N)\,.
\end{equation}
Now $\log\log\log N \ge \log\log 983$ and $-\log \sigma/\log\log N \le 1/5$ so by Lemma \ref{inverse} and inequality \eqref{epsilon2}
\begin{align}
z(\sigma \log N) &=
\frac{\sigma \log N}{\log \left( \sigma \log N\right)} \left( 1+ \frac{\log\log z(\sigma \log N)}{\log z(\sigma \log N)} \right) \notag 
\\  & \le
\frac{\sigma \log N}{\log\log N} \left( 1+ 2\varepsilon(N) \right)\left( 1-\frac{5\log \sigma}{4\log\log N}\right)\notag \\  & \le
\frac{\sigma \log N}{\log\log N} \left( 1+ 2\varepsilon(N) \right)\left( 1-\frac{5\log \sigma}{4\log\log 983}\varepsilon(N)\right)\notag \\& \le
\frac{\sigma \log N}{\log\log N} \left( 1+ \tfrac{9}{10}\varepsilon(N)+2\varepsilon(N) +\tfrac{9}{5}\varepsilon(N)^2\right) \notag \\ & \le
\frac{\sigma \log N}{\log\log N} \left( 1+ \tfrac{7}{2}\varepsilon(N) \right)\,. \label{ztoz12}
\end{align}
Then
\begin{align*}
\zeta(N) &  \ge \frac{z_1(\sigma \log N)}{\sigma}+\beta 
\ge \frac{(\log N)^{1/2}}{\sigma}+\beta  \ge \frac{\sqrt{983}}{\sigma}+1\ge 111\,
\end{align*}
and
\begin{align}
Z(N)
&\le \zeta(N)^{3} \bigg(\frac{8\cdot 9 \cdot 27(1+1/111)(4+15/111)}{3\sqrt{3}}+\frac{3(1+1/111)}{e^{2\cdot982}111^2(e^3-1)}\bigg) \notag\\
&\leq 1561\zeta(N)^{3} \,. \label{booboo}
\end{align}
Inequality \eqref{apuepis2} implies 
\begin{equation}\label{valivaihe2}
(\log N)^{2/5}\log\log N< \frac{\log N}{\log\log N}=\varepsilon(N)\frac{\log N}{\log\log\log N}\, .
\end{equation}
Since $\log N\geq 983$ then \eqref{valivaihe2} gives
$30\log\log N < \log N \varepsilon(N)$. Then by \eqref{ztoz12} we have
\begin{align}
\zeta(N) & = \frac{z(\sigma \log N)}{\sigma} +\beta < \frac{\log N}{\log\log N}
\left(1+\tfrac{7}{2}\varepsilon(N) \right)+1 \notag \\ & <
\frac{\log N}{\log\log N}
\left(1+(\tfrac{7}{2}+\tfrac{1}{30})\varepsilon(N) \right)<
\frac{\log N}{\log\log N}
\left(1+\tfrac{11}{3}\varepsilon(N) \right)\,. \label{zeta2}
\end{align}
Finally since \eqref{apuepis2} implies $(\log N)^{7/10}<\log N/\log\log N$ we can estimate
\begin{align*}
3\log\left(1+\tfrac{11}{3}\varepsilon(N)\right)&<11\varepsilon(N)=11\varepsilon(N)\log 3(\log N)^{7/10} (\log 3)^{-1}(\log N)^{-7/10}\\
&<11 (\log 3)^{-1} 983^{-7/10}\varepsilon(N)\log 3\frac{\log N}{\log\log N}\\
&<\tfrac{1}{10}\varepsilon(N)\log 3\frac{\log N}{\log\log N}
\end{align*}
and get
\begin{equation}\label{epsiloncontrol2}
\left(1+\tfrac{11}{3}\varepsilon(N)\right)^{3} \le N^{\log3\varepsilon(N)/(10\log\log N)}\,.
\end{equation}
Using Theorem \ref{GENERALTHEOREM} and inequalities \eqref{ztoz12}, \eqref{booboo}, \eqref{zeta2}
and then \eqref{epsiloncontrol2} gives
\begin{align*}
1&<\left|e^{s/t}-\frac{M}{N}\right|Z(N)\,
N^{2+2\log(|s|/\alpha)z(\sigma\log N)/(\sigma\log N)}\\
&<\left|e^{s/t}-\frac{M}{N}\right| 1561 \,
N^{2+\log 3(1+7/2\varepsilon(N))/\log\log N}\left(\frac{\log N}{\log\log N}\left( 1+\tfrac{11}{3}\varepsilon(N)\right)\right)^{3} \\ &
<\left|e^{s/t}-\frac{M}{N}\right|  1561 \,
N^{2+\log3(1+4\varepsilon(N))/\log\log N}\left(\frac{\log N}{\log\log N}\right)^{3}\,,
\end{align*}
which is what had to be proved.
\end{proof}

   Kalle Lepp\"al\"a\\   
   Bioinformatics Research Centre,\\
   Aarhus University
   C.F. M{\o}llers All\'e 8
   DK-8000 Aarhus C\\
   Denmark \\
   kalle.m.leppala@gmail.com\\
  
   Tapani Matala-aho and Topi T\"orm\"a\\
   Mathematics,\\    
   University of Oulu, P.O. Box 3000, 90014 Oulu\\
   Finland 
   tapani.matala-aho@oulu.fi\\
   topi.torma@oulu.fi\\
   
\end{document}